\documentclass[reqno]{amsart}
\usepackage{color,hyperref}
 \newtheorem{theorem}{Theorem}[section]
 \newtheorem{corollary}[theorem]{Corollary}
 \newtheorem{lemma}[theorem]{Lemma}
\newcommand{\C}{\mathbb{C}}
\newcommand{\N}{\mathbb{N}}
\newcommand{\cL}{\mathcal{L}}
\newcommand{\cM}{\mathcal{M}}
\newcommand{\cP}{\mathcal{P}}
\newcommand{\eps}{\varepsilon}
\newcommand{\Image}{\operatorname{Im}}
\newcommand{\Ker}{\operatorname{Ker}}
\begin{document}
\title[The Coburn-Simonenko theorem for Toeplitz operators]{%
The Coburn-Simonenko theorem for Toeplitz operators acting between Hardy type 
subspaces of different Banach function spaces}
\author{Alexei Yu. Karlovich}
\address{%
A. Yu. Karlovich,
Centro de Matem\'atica e Aplica\c{c}\~oes,
Departamento de Matem\'a\-tica, Faculdade de Ci\^encias e Tecnologia,
Universidade Nova de Lisboa,
Quinta da Torre, 2829--516 Caparica, Portugal}
\email{oyk@fct.unl.pt}
\thanks{%
This work was partially supported by the Funda\c{c}\~ao para a Ci\^encia 
e a Tecnologia (Portuguese Foundation for Science and Technology)
through the project UID/MAT/00297/2013 (Centro de Matem\'atica e 
Aplica\c{c}\~oes). 
}
\begin{abstract}
Let $\Gamma$ be a rectifiable Jordan curve, let $X$ and $Y$ be two reflexive 
Banach function spaces over $\Gamma$ such that the 
Cauchy singular  integral operator $S$ is bounded on each of them, and let
$M(X,Y)$ denote the space of pointwise multipliers from $X$ to $Y$.
Consider the Riesz projection $P=(I+S)/2$, the corresponding Hardy type 
subspaces $PX$ and $PY$, and the Toeplitz operator $T(a):PX\to PY$ defined 
by $T(a)f=P(af)$ for a symbol $a\in M(X,Y)$. We show that if 
$X\hookrightarrow Y$ and 
$a\in M(X,Y)\setminus\{0\}$, then $T(a)\in\cL(PX,PY)$ has a trivial kernel
in $PX$ or a dense image in $PY$. In particular, if $1<q\le p<\infty$, 
$1/r=1/q-1/p$, and $a\in L^{r}\equiv M(L^p,L^q)$ is a nonzero function, 
then the Toeplitz operator $T(a)$, acting from the Hardy space $H^p$ to 
the Hardy space $H^q$, has a trivial kernel in $H^p$ or a dense image in $H^q$.
\end{abstract}
\keywords{%
Toeplitz operator, 
symbol, 
Banach function space, 
variable Lebesgue space,
pointwise mutiplier, 
Coburn-Simonenko theorem}
\subjclass[2010]{47B35, 46E30}
\maketitle
\section{Introduction}
Let $\Gamma$ be a Jordan curve, that is, a curve that homeomorphic to a circle. 
We suppose that $\Gamma$ is rectifiable and equip it with the Lebesgue length 
measure $|d\tau|$ and the counter-clockwise orientation. The Cauchy singular 
integral of a measurable function $f:\Gamma\to\C$ is defined by
\begin{equation}\label{eq:def-S}
(Sf)(t):=\lim_{\eps\to 0}\frac{1}{\pi i}\int_{\Gamma\setminus\Gamma(t,\eps)}
\frac{f(\tau)}{\tau-t}d\tau,
\quad t\in\Gamma,
\end{equation}
where the ``portion'' $\Gamma(t,\eps)$ is
\[
\Gamma(t,\eps):=\{\tau\in\Gamma:|\tau-t|<\eps\},
\quad \eps>0.
\]
It is well known that $(Sf)(t)$ exists a.e. on $\Gamma$ whenever $f$ is 
integrable (see \cite[Theorem~2.22]{Dynkin87}).

For two normed spaces $X$ and $Y$, we will write $X\hookrightarrow Y$ 
if there is a constant $c\in(0,\infty)$ such that $\|f\|_Y\le c\|f\|_X$ 
for all $f\in X$, $X=Y$ if $X$ and $Y$ coincide as sets and
there are constants $c_1,c_2\in(0,\infty)$ such that 
$c_1\|f\|_X\le \|f\|_Y\le c_2\|f\|_X$ for all $f\in X$, and 
$X\equiv Y$ if $X$ and $Y$ coincide as sets and $\|f\|_X=\|f\|_Y$
for all $f\in X$. As usual, the space of all bounded linear operators
from $X$ to $Y$ is denoted by $\cL(X,Y)$. We adopt the standard 
abbreviation $\cL(X)$ for $\cL(X,X)$.

Let $\gamma$ be a measurable subset of $\Gamma$ of positive measure.
The set of all measurable complex-valued functions on $\gamma$ is denoted by 
$\cM(\gamma)$. Let $\cM^+(\gamma)$ be the subset of functions in $\cM(\gamma)$ 
whose values lie in $[0,\infty]$. The characteristic function of a measurable 
set $E\subset\gamma$ is denoted by $\chi_E$.

Following \cite[Chap.~1, Definition~1.1]{BS88}, a mapping 
$\rho_\gamma:\cM^+(\gamma)\to [0,\infty]$ is called a Banach function norm
if, for all functions $f,g, f_n\in\cM^+(\gamma)$ with $n\in\N$, for all
constants $a\ge 0$, and for all measurable subsets $E$ of $\gamma$, the
following  properties hold:
\begin{eqnarray*}
{\rm (A1)} & &
\rho_\gamma(f)=0  \Leftrightarrow  f=0\ \mbox{a.e.},
\quad
\rho_\gamma(af)=a\rho_\gamma(f),
\quad
\rho_\gamma(f+g) \le \rho_\gamma(f)+\rho_\gamma(g),\\
{\rm (A2)} & &0\le g \le f \ \mbox{a.e.} \ \Rightarrow \ 
\rho_\gamma(g) \le \rho_\gamma(f)
\quad\mbox{(the lattice property)},\\
{\rm (A3)} & &0\le f_n \uparrow f \ \mbox{a.e.} \ \Rightarrow \
       \rho_\gamma(f_n) \uparrow \rho_\gamma(f)\quad\mbox{(the Fatou property)},\\
{\rm (A4)} & &\rho_\gamma(\chi_E) <\infty,\\
{\rm (A5)} & &\int_E f(\tau)|d\tau| \le C_E\rho_\gamma(f)
\end{eqnarray*}
with the constant $C_E \in (0,\infty)$ that may depend on $E$ and 
$\rho_\gamma$,  but is independent of $f$. When functions differing only on 
a set of measure  zero are identified, the set $X(\gamma)$ of all functions 
$f\in\cM(\gamma)$ for  which  $\rho(|f|)<\infty$ is called a Banach function
space. For each $f\in X(\gamma)$, the norm of $f$ is defined by
\[
\|f\|_X(\gamma) :=\rho(|f|).
\]
The set $X(\gamma)$ under the natural linear space operations and under 
this norm becomes a Banach space (see 
\cite[Chap.~1, Theorems~1.4 and~1.6]{BS88}) and
\[
L^\infty(\gamma)\hookrightarrow X(\gamma)\hookrightarrow L^1(\gamma).
\]
If $\rho_\gamma$ is a Banach function norm, its associate norm 
$\rho'_\gamma$ is defined on $\cM^+(\gamma)$ by
\[
\rho_\gamma'(g):=\sup\left\{
\int_\gamma f(\tau)g(\tau)|d\tau| \ : \ 
f\in \cM^+(\gamma), \ \rho_\gamma(f) \le 1
\right\}, \quad g\in \cM^+(\gamma).
\]
It is a Banach function norm itself \cite[Chap.~1, Theorem~2.2]{BS88}.
The Banach function space $X'(\gamma)$ determined by the Banach function norm
$\rho'_\gamma$ is called the associate space (K\"othe dual) of $X(\gamma)$. 
The associate space $X'(\gamma)$ can be viewed a subspace of the dual space
$X^*(\gamma)$. 

Recall that, since the Lebesgue length measure $|d\tau|$ is separable
(see, e.g., \cite[Section~6.10]{KA84}), a Banach function space $X(\gamma)$ 
over $\gamma$ is separable if and only if its
K\"othe dual space $X'(\gamma)$ is isometrically isomorphic to the Banach dual
space $X^*(\gamma)$ (see, e.g., \cite[Chap.~1, Corollaries 4.3, 4.4]{BS88}).
A Banach function space $X(\gamma)$ reflexive if and only if $X(\gamma)$ and  
$X'(\gamma)$ are  separable (see, e.g., \cite[Chap.~1, Corollary~5.6]{BS88}).

For Banach function spaces $X(\gamma)$ and $Y(\gamma)$, let
$M(X(\gamma),Y(\gamma))$ denote the space of pointwise multipliers from
$X(\gamma)$ to $Y(\gamma)$ defined by 
\[
M(X(\gamma),Y(\gamma)):=\{f\in\cM(\gamma)\ :\ 
fg\in Y(\gamma)\text{ for all } g\in X(\gamma)\}.
\]
It is a Banach function space with respect to the operator norm
\[
\|f\|_{M(X(\gamma),Y(\gamma))}=\sup\{\|fg\|_{Y(\gamma)}\ :\ 
g\in X(\gamma),\ \|g\|_{X(\gamma)}\le 1\}.
\]
In particular, $M(X(\gamma),X(\gamma))\equiv L^\infty(\gamma)$.
Note that it may happen that the space $M(X(\gamma),Y(\gamma))$ contains 
only the zero function.
For instance, if $1\le p<q<\infty$, then $M(L^p(\gamma),L^q(\gamma))=\{0\}$.
The continuous embedding 
$L^\infty(\gamma)\hookrightarrow M(X(\gamma),Y(\gamma))$ 
holds if and only if $X(\gamma)\hookrightarrow Y(\gamma)$. For example, if 
$1\le q\le p\le\infty$, then $L^p(\gamma)\hookrightarrow L^q(\gamma)$ and 
$M(L^p(\gamma),L^q(\gamma))\equiv L^{r}(\gamma)$, where $1/r=1/q-1/p$. 
For these and many other properties and examples, we refer to 
\cite{KLM13,LT16,MN10,MP89,N16} (see also references therein).

For the brevity, we will write $X:=X(\Gamma)$ if $\Gamma$ is a 
rectifiable Jordan curve.
If $X$ is a reflexive Banach function space over a rectifiable Jordan 
curve $\Gamma$ and the Cauchy singular integral operator defined by 
\eqref{eq:def-S} is bounded on $X$, then in view of 
\cite[Theorem~6.1]{K03} and the H\"older inequality for Banach function
spaces (see, e.g., \cite[Chap.~1, Theorem~2.4]{BS88}), the curve 
$\Gamma$ is a Carleson curve (or Ahlfors-David regular curve), that is, 
\[
\sup_{t\in\Gamma}\sup_{\eps>0}\frac{|\Gamma(t,\eps)|}{\eps}<\infty.
\]
Moreover, by \cite[Lemma~6.4]{K03}, the operators
\[
P:=(I+S)/2,\quad Q:=(I-S)/2
\]
are bounded projections both on $X$ and on $X'$, the latter means that 
$P^2=P$ and $Q^2=Q$. Then we can define Hardy type subspaces $PX,QX$ of 
$X$ and $PX',QX'$ of $X'$.

In what follows we will always assume that $X$ and $Y$ are reflexive
Banach function spaces and $S$ is bounded on both $X$ and $Y$. For 
$a\in M(X,Y)$, define the Toeplitz operator $T(a):PX\to PY$ with symbol $a$ by
\[
T(a)f=P(af),
\quad f\in PX.
\]
It is clear that $T(a)\in\cL(PX,PY)$ and 
\[
\|T(a)\|_{\cL(PX,PY)}\le\|P\|_{\cL(Y)}\|a\|_{M(X,Y)}.
\]

We note that there is a huge literature dedicated to Toeplitz operator
acting between the same Hardy spaces $H^p=PL^p$, $1<p<\infty$,
see, e.g., the monographs by 
Douglas \cite{D98},
B\"ottcher and Silbermann \cite{BS06},
Gohberg, Goldberg, Kaashoek \cite{GGK90},
Nikolski \cite{N02} 
for Toeplitz operators on Hardy spaces over the unit circle and the monograph 
by B\"ottcher and Karlovich \cite{BK97} for Toeplitz operators on weighted
Hardy spaces over Carleson curves. 

Surprisingly enough, we could find only one paper by Tolokonnikov \cite{T87} 
dedicated to Toeplitz operators acting between different Hardy spaces
$H^p$ and $H^q$ over the unit circle. In particular, he described in 
\cite[Theorem~4]{T87} all symbols generating bounded Toeplitz operators from 
$H^p$ to $H^q$ for $0<p,q\le\infty$. Very recently, Le\'snik \cite{L17}
proposed to study Toeplitz and Hankel operators between abstract Hardy spaces
$H[X]$ and $H[Y]$ built upon different separable rearrangement-invariant 
Banach function spaces $X$ and $Y$ over the unit circle such that 
$X\hookrightarrow Y$ and the space $Y$ has nontrivial Boyd indices. Notice 
that the latter condition 
is equivalent to the boundedness of the operator $S$ on the space 
$Y$, whence $H[Y]=PY$. Le\'snik obtained analogues of the Brown-Halmos and 
Nehari theorems (see \cite[Theorem~4.2]{L17} and \cite[Theorem~5.5]{L17}, 
respectively), extending results of the author \cite{K04} for the case 
of a reflexive rearrangement-invariant Banach function space $X$ 
(that is, $X=Y$) with nontrivial Boyd indices. He also proved 
\cite[Theorem~6.1]{L17} that a Toeplitz operator $T(a):H[X]\to H[Y]$ is 
compact if and only if $a=0$.

Inspired by the work of Le\'snik \cite{L17}, we prove the following analogue 
of the Coburn-Simonenko theorem for Toeplitz operators $T(a):PX\to PY$ in the 
case when $X$ and $Y$ are different Banach function spaces. Notice that
we do not assume that the spaces $X$ and $Y$ are rearrangement-invariant.
\begin{theorem}\label{th:main}
Let $X$ and $Y$ be reflexive Banach function spaces over a rectifiable Jordan
curve $\Gamma$. Suppose $X\hookrightarrow Y$ and the Cauchy singular 
integral operator $S$ given by \eqref{eq:def-S} is bounded on $X$ and on $Y$. 
If $a\in M(X,Y)\setminus\{0\}$, then $T(a)\in\cL(PX,PY)$ has a trivial kernel 
in $PX$ or a dense image in $PY$.
\end{theorem} 
The above result was proved by Coburn \cite{C66} for the case of $X=Y=L^2$ 
over the unit circle and  by Simonenko \cite{S68} in a more general of setting 
of $X=Y=L^p$, $1<p<\infty$, over so-called Lyapunov curves. We also refer to 
\cite[Theorem~6.17]{BK97}, where the above theorem is proved in the case 
$X=Y=L^p(w)$, where $L^p(w)$, $1<p<\infty$, is a Lebesgue space with a 
Muckenhoupt weight over a Carleson Jordan curve.

The statement of Theorem~\ref{th:main} has a more precise form for concrete 
Banach function spaces $X,Y$ when $M(X,Y)$ can be calculated and conditions 
for the boundedness of $S$ are known. Here we mention only the case of Toeplitz 
operators acting from the Hardy space $H^p=PL^p$ to the Hardy space $H^q=PL^q$ 
as the  simplest example. 
\begin{corollary}\label{co:Hardy}
Let $1<q\le p<\infty$ and $1/r=1/q-1/p$. Suppose $\Gamma$ is a Carleson  
Jordan curve. If $a\in L^r\setminus\{0\}$, then the Toeplitz operator
$T(a)\in\cL(H^p,H^q)$ has a trivial kernel in $H^p$ or a dense image in $H^q$.
\end{corollary}
It seems that the above corollary is new even in the case of the unit circle. 

The paper is organized as follows. In Section~\ref{sec:preliminaries}, we 
collect properties of Banach function spaces and their Hardy type subspaces
proved elsewhere. In Section~\ref{sec:proofs}, we first relate the triviality
of the kernel (resp. the density of the image) of a Toeplitz operator
$T(a)\in\cL(PX,PY)$ with the density of the range (resp. triviality of the 
kernel) of its companion operator $\widetilde{T}(a):\cL(QY',QX')$ defined 
by $\widetilde{T}(a)f=Q(af)$. Then show that one of the operators $T(a)$ 
or $\widetilde{T}(a)$ is injective with the aid of the Lusin-Privalov 
theorem and other results stated in Section~\ref{sec:preliminaries}. 
In Section~\ref{sec:VLS} we recall the definition of variable Lebesgue
spaces $L^{p(\cdot)}$, which give a non-trivial example of Banach function 
spaces. Further, we describe the space $M(L^{p(\cdot)},L^{r(\cdot)})$
and formulate conditions for the boundedness of the operator Cauchy
singular operator $S$ on $L^{p(\cdot)}$. These results allow us to reformulate
Theorem~\ref{th:main} for Toeplitz operators between $PL^{p(\cdot)}$
and $PL^{q(\cdot)}$ in terms of variable exponents $p,q:\Gamma\to(1,\infty)$.
In particular, we immediately get Corollary~\ref{co:Hardy}, taking
all exponents constant.
\section{Preliminaries}\label{sec:preliminaries}
\subsection{The Lusin-Privalov theorem}
Let $\Gamma$ be a rectifiable Jordan curve. It divides the plane into a bounded
connected component $D^+$ and an unbounded connected component $D^-$. We
provide $\Gamma$ with the counter-clockwise orientation, that is, we demand
that $D^+$ stays on the left of $\Gamma$ when the curve is traced out in the
positive direction. Without loss of generality we suppose that $0\in D^+$.
Put
\begin{eqnarray*}
L^1_+ &:=& \left\{ f\in L^1:\quad\int_\Gamma f(\tau)\tau^n d\tau=0
\quad\mbox{for}\quad n\ge 0\right\},
\\
(L^1)_-^0 &:=& \left\{ f\in L^1:\quad\int_\Gamma f(\tau)\tau^n d\tau=0
\quad\mbox{for}\quad n<0\right\},
\\
L^1_- &:=& (L^1)_-^0 \oplus\C.
\end{eqnarray*}
From \cite[pp.~202--206]{P50} one can extract the following result.
\begin{lemma}\label{le:Privalov}
We have $L_+^1\cap (L^1)_-^0=\{0\}$ and  $L^1_+\cap L^1_-=\C$.
\end{lemma}
The proof of the following important theorem  is contained in
\cite[p.~292]{P50} or \cite[Theorem~10.3]{D70}.
\begin{theorem}[Lusin-Privalov]\label{th:Lusin-Privalov}
Let $\Gamma$ be a rectifiable Jordan curve. If $f\in L_\pm^{1}$, then $f$ 
vanishes either almost everywhere on $\Gamma$ or almost nowhere on $\Gamma$.
\end{theorem}
\subsection{Properties of Banach function spaces and pointwise multipliers}
In this subsection we collect some well known properties of Banach function
spaces and pointwise multipliers between them. 
\begin{lemma}[{\cite[Chap.~1, Proposition~2.10]{BS88}}]
\label{le:duality-embedding}
Let $X,Y$ be Banach function spaces over a rectifiable Jordan curve $\Gamma$ 
and let $X',Y'$ be their associate spaces, respectively. If  
$X\hookrightarrow Y$, then $Y'\hookrightarrow X'$.
\end{lemma}
\begin{lemma}[{\cite[Section~2, property (vii)]{KLM13}}]
\label{le:properties-multipliers}
Let $X,Y$ be Banach function spaces over a rectifiable Jordan curve $\Gamma$
and let $X',Y'$ be their associate spaces, respectively. Then 
$M(X,Y)\equiv M(Y',X')$.
\end{lemma}
\begin{lemma}\label{le:multiplication-adjoint}
Let $X,Y$ be separable Banach function spaces over a rectifiable Jordan curve
$\Gamma$ and $a\in M(X,Y)$. Then the adjoint of the operator $aI\in\cL(X,Y)$ 
of multiplication by the function $a$ is the operator 
$(aI)^*=\overline{a}I\in\cL(Y',X')$. 
\end{lemma}
\begin{proof}
Since $X$ (resp., $Y$) is separable, its Banach dual space $X^*$ (resp., $Y^*$) 
is isometrically isomorphic to the the associate (K\"othe dual) space $X'$ 
(resp., $Y'$) and
\[
G(f)=\int_\Gamma f(\tau)\overline{g(\tau)}|d\tau|
\]
gives the general form of a linear functional on $X$ (resp., $Y$) and 
$\|G\|_{X^*}=\|g\|_{X'}$ (resp., $\|G\|_{Y^*}=\|g\|_{Y'}$),
see, e.g, \cite[Chap.~1, Corollary 4.3]{BS88}. The desired statement follows
immediately from the above observation and 
Lemma~\ref{le:properties-multipliers}.
\end{proof}
\subsection{Hardy type subspaces of a Banach function space}
Suppose $X$ is a reflexive Banach function space in which the Cauchy
singular integral operator $S$ is bounded. Put
\[
X_+:=PX,
\quad X_-^0:=QX,\quad X_-:=X_-^0\oplus\C.
\]
The corresponding subspaces $X_+'$, $(X')_-^0$, $X_-'$ are defined analogously.

For $f\in X\subset L^1$, consider the Cauchy type integrals
\[
(C_\pm f)(z):=\frac{1}{2\pi i}\int_\Gamma\frac{f(\tau)}{\tau-z}d\tau,
\quad z\in D^\pm.
\]
It is well known \cite[p. 189]{P50} that the functions $(C_\pm f)(z)$ are 
analytic in $D^\pm$, they have nontangential boundary values $(C_\pm f)(t)$ as 
$z\to t$ almost everywhere on $\Gamma$. These boundary values can be found by
the Sokhotsky-Plemelj formulas
\[
(C_\pm f)(t)=
\frac{1}{2}f(t)\pm \frac{1}{2\pi i}\int_\Gamma\frac{f(\tau)}{\tau-t}d\tau,
\]
that is,
\[
(C_+ f)(t)=(P f)(t),\quad (C_-f)(t)=(Qf)(t).
\]
Since the function $f\in X_+$ (respectively, $f\in X_-^0$) coincides on 
$\Gamma$ with the boundary value of the function $C_+f$ (respectively, $C_-f$) 
defined in $D^+$ (respectively, $D^-$), we will think of functions from $X_+$ 
(respectively, $X_-^0$) as of functions defined in $D^+$ (respectively, in 
$D^-$) by $f(z):=(C_+f)(z)$ (respectively, by $f(z):=(C_-f)(z)$).
\begin{lemma}[{\cite[Lemma~6.9]{K03}}]
\label{le:subspaces-properties}
Let $\Gamma$ be a rectifiable Jordan curve and $X$ be a reflexive Banach 
function space in which the Cauchy singular integral operator $S$ is bounded.
\begin{enumerate}
\item[{\rm (a)}] 
If $f\in X_\pm$ and $g\in X_\pm'$, then $fg\in L^1_\pm$.
If, in addition, $f\in X_-^0$ or $g\in (X')_-^0$, then
$fg\in (L^1)_-^0$.

\item[{\rm (b)}] We have
\[
X_+=L^1_+\cap X,\quad
X_-^0=(L^1)_-^0\cap X,\quad
X_-=L^1_-\cap X.
\]
\end{enumerate}
\end{lemma}
\subsection{Adjoint operators of the projections $P$ and $Q$}
On a rectifiable Jordan oriented curve $\Gamma$, we have 
\[
d\tau=e^{i\theta_\Gamma(\tau)}|d\tau|,
\]
where $\theta_\Gamma(\tau)$ is the angle made by the positively oriented real
axis and the naturally oriented tangent of $\Gamma$ at $\tau$ (which exists
almost everywhere). Let $X$ be a Banach function space over $\Gamma$. Define
the operator $H_\Gamma:X\to X$ by 
\[
(H_\Gamma f)(\tau):=e^{-i\theta_\Gamma(\tau)}\overline{f(\tau)}.
\]
Note that the operator $H_\Gamma$ is additive but 
$H_\Gamma(\alpha f)=\overline{\alpha}\cdot H_\Gamma f$ for $\alpha\in\C$ 
and $f\in X$. It is clear that $H_\Gamma$ is bounded on $X$ and $H_\Gamma^2=I$.
\begin{lemma}[{\cite[Lemma~6.6]{K03}}]
\label{le:projections-adjoints}
Let $\Gamma$ be a rectifiable Jordan curve and $X$ be a reflexive Banach 
function space in which the Cauchy singular integral operator $S$ is bounded.
Then the adjoint of $S\in\cL(X)$ is the operator 
$S^*=-H_\Gamma SH_\Gamma\in\cL(X')$ and consequently,
\[
P^*=H_\Gamma QH_\Gamma,
\quad
Q^*=H_\Gamma PH_\Gamma.
\]
\end{lemma}
\section{Proof of the main results}\label{sec:proofs}
\subsection{Companion operator of a Toeplitz operator}
Let $X$ and $Y$ be reflexive Banach function spaces over a rectifiable
Jordan curve $\Gamma$. Suppose $a\in M(X,Y)\equiv M(Y',X')$ and the operator 
$S$ is bounded on $X$ and on $Y$. In view of 
Lemma~\ref{le:projections-adjoints}, the operator $S$ is also bounded on $Y'$ 
and on $X'$. Then, along with the Toeplitz operator $T(a):X_+\to Y_+$, we 
consider its companion operator $\widetilde{T}(a):(Y')_-^0\to (X')_-^0$ defined
by
\[
\widetilde{T}(a)f=Q(af),\quad f\in (Y')_-^0.
\]
It is obvious that $\widetilde{T}(a)\in\cL((Y')_-^0,(X')_-^0)$ and
\[
\|\widetilde{T}(a)\|_{\cL((Y')_-^0,(X')_-^0)}\le\|Q\|_{\cL(X')}\|a\|_{M(X,Y)}.
\]
\begin{lemma}\label{le:Toeplitz-companion}
Let $X$ and $Y$ be reflexive Banach function spaces over a rectifiable
Jordan curve. Suppose $X\hookrightarrow Y$ and the Cauchy singular 
integral operator $S$ given by \eqref{eq:def-S} is bounded on $X$ and on $Y$.
If $a\in M(X,Y)$, then the Toeplitz operator $T(a):X_+\to Y_+$ has a trivial
kernel in $X_+$ (resp., a dense image in $Y_+$) if and only if its
companion operator $\widetilde{T}(a):(Y')_-^0\to (X')_-^0$ has a dense
image in $(X')_-^0$ (resp., a trivial kernel in $(Y')_-^0$).
\end{lemma}
\begin{proof}
Let $\Image A$ and $\Ker A$ denote the image and the kernel, respectively, of 
a bounded linear operator $A$ acting between Banach spaces.

Since $X\hookrightarrow Y$, we have $Q\in\cL(X,Y)$ and $PaP+Q\in\cL(X,Y)$.
The spaces $X$ and $Y$ decompose into the direct sums $X=X_+\oplus X_-^0$
and $Y=Y_+\oplus Y_-^0$. Accordingly, the operator $PaP+Q$ may be written as
an operator matrix
\[
\begin{pmatrix}
T(a) & 0 \\ 0 & I
\end{pmatrix}
:
\begin{pmatrix}
X_+ \\ X_-^0
\end{pmatrix}
\to
\begin{pmatrix}
Y_+ \\ Y_-^0
\end{pmatrix}.
\]
Hence
\begin{equation}\label{eq:Toeplitz-companion-1}
\Image(PaP+Q)=\Image T(a)\oplus Y_-^0,
\quad
\Ker (PaP+Q)=\Ker T(a).
\end{equation}

On the other hand, $Y'\hookrightarrow X'$ by Lemma~\ref{le:duality-embedding}
and $a\in M(Y',X')$ by Lemma~\ref{le:properties-multipliers}.
Then $P\in\cL(Y',X')$ and $P+QaQ\in\cL(Y',X')$. Since the spaces $Y'$ and $X'$
decompose into the direct sums $Y'=(Y')_+\oplus (Y')_-^0$ and
$X'=(X')_+\oplus (X')_-^0$, the operator $P+QaQ$ may be written as an operator
matrix
\[
\begin{pmatrix}
I & 0 \\ 0 & \widetilde{T}(a)
\end{pmatrix}
:
\begin{pmatrix}
(Y')_+ \\ (Y')_-^0
\end{pmatrix}
\to
\begin{pmatrix}
(X')_+ \\ (X')_-^0
\end{pmatrix}.
\]
Therefore
\begin{equation}\label{eq:Toeplitz-companion-2}
\Image(P+QaQ)=(X')_+\oplus\Image \widetilde{T}(a),
\quad
\Ker (P+QaQ)=\Ker \widetilde{T}(a).
\end{equation}

Lemmas~\ref{le:multiplication-adjoint} and~\ref{le:projections-adjoints}
yield
\begin{align}
(PaP+Q)^*
&=
P^*\overline{a}P^*+Q^*
=
(H_\Gamma QH_\Gamma)(H_\Gamma a H_\Gamma)(H_\Gamma QH_\Gamma)+H_\Gamma PH_\Gamma
\nonumber\\
&=
H_\Gamma(P+QaQ)H_\Gamma.
\label{eq:Toeplitz-companion-3}
\end{align}

From the second identity in \eqref{eq:Toeplitz-companion-1} it follows that
$T(a)\in\cL(X_+,Y_+)$ has a trivial kernel in $X_+$ if and only if
$PaP+Q\in\cL(X,Y)$ has a trivial kernel in $X$. On the other hand, from
\eqref{eq:Toeplitz-companion-3} and $H_\Gamma^2=I$ we deduce that the latter
fact is equivalent to the fact that $P+QaQ\in\cL(Y',X')$ has a dense
image in $X'$ (see, e.g., \cite[Section~4.12]{R91}). In turn, in view of
the first identity in \eqref{eq:Toeplitz-companion-2}, the operator
$P+QaQ$ has a dense image in $X'$ if and only if the operator 
$\widetilde{T}(a)\in\cL((Y')_-^0,(X')_-^0)$ has a dense image in 
$(X')_-^0$.

The proof of the equivalence of the density of the image of $T(a)$ in $Y_+$
and the triviality of the kernel of $\widetilde{T}(a)$ in $(Y')_-^0$ is 
analogous.
\end{proof}
\subsection{Proof of Theorem~\ref{th:main}}	
In view of Lemma~\ref{le:Toeplitz-companion}, it is sufficient to show that 
$T(a):X_+\to Y_+$ is injective on $X_+$ or 
$\widetilde{T}(a):(Y')_-^0\to (X')_-^0$ is injective on $(Y')_-^0$.

Assume the contrary, that is, that there exist $f_+\in X_+$ and
$g_-\in (Y')_-^0$ such that $f_+\ne 0$, $g_-\ne 0$, and
\begin{equation}\label{eq:proof-1}
Paf_+=0,
\quad
Qag_-=0.
\end{equation}
By Lemma~\ref{le:subspaces-properties}(b), $f_+\in X_+\subset L_+^1$ and 
$g_-\in (Y')_0^-\subset L_-^1$. Since $f_+\ne 0$ and $g_-\ne 0$, from
the Lusin-Privalov Theorem~\ref{th:Lusin-Privalov} it follows that
$f_+\ne 0$ a.e. on $\Gamma$ and $g_-\ne 0$ a.e. on $\Gamma$.

Put $f_-:=af_+$ and $g_+:=ag_-$. Then from \eqref{eq:proof-1} it follows that
$Paf_+=Pf_-=0$ and $Qag_-=Qg_+=0$. Therefore,
\begin{align*}
f_-
&=
af_+=Paf_++Qaf_+=Qaf_+\in Y_-^0,
\\
g_+
&=
ag_-=Pag_-+Qag_-=Pag_-\in (X')_+.
\end{align*}
Then
\begin{equation}\label{eq:proof-2}
f_+g_+=f_+(ag_-)=(f_+a)g_-=f_-g_-.
\end{equation}
From Lemma~\ref{le:subspaces-properties}(a) we deduce that $f_+g_+\in L_+^1$
and $f_-g_-\in (L^1)_-^0$. Lemma~\ref{le:Privalov} and identity
\eqref{eq:proof-2} imply that $f_+g_+=f_-g_-=f_+ag_-=0$. Since
$f_+\ne 0$ a.e. on $\Gamma$ and $g_-\ne 0$ a.e. on $\Gamma$,
we conclude that $a=0$ a.e. on $\Gamma$, but this contradicts our hypothesis 
and, thus, completes the proof.
\qed
\section{Toeplitz operators between Hardy type subspaces\\ of variable Lebesgue spaces}
\label{sec:VLS}
\subsection{Variable Lebesgue spaces}
Given a rectifiable Jordan curve $\Gamma$ , let $\cP(\Gamma)$ be the set of all 
measurable functions $p:\Gamma\to[1,\infty]$. For $p\in\cP(\Gamma)$
and a measurable subset $\gamma\subset\Gamma$, put
\[
\gamma_\infty^{p(\cdot)} :=\{t\in\gamma: p(t)=\infty\}.
\]
For a measurable function $f:\gamma\to\C$, consider
\[
\varrho_{p(\cdot),\gamma}(f)
:=
\int_{\gamma\setminus\gamma_\infty^{p(\cdot)}}|f(t)|^{p(t)}|dt|
+\|f\|_{L^\infty(\gamma_\infty^{p(\cdot)})}.
\]
According to \cite[Definition~2.9]{CF13}, the variable Lebesgue space
$L^{p(\cdot)}(\gamma)$ is defined as the set of all measurable functions
$f:\gamma\to\C$ such that $\varrho_{p(\cdot),\gamma}(f/\lambda)<\infty$
for some $\lambda>0$. This space is a Banach function space with respect
to the Luxemburg-Nakano norm given by
\[
\|f\|_{L^{p(\cdot)}(\gamma)}:=\inf\{\lambda>0: \varrho_{p(\cdot),\gamma}(f/\lambda)\le 1\}
\]
(see, e.g., \cite[Theorems~2.17, 2.71 and Section~2.10.3]{CF13}). 
If $p\in\cP(\Gamma)$ is constant, then $L^{p(\cdot)}(\gamma)$ is nothing
but the standard Lebesgue space $L^p(\gamma)$.
Variable Lebesgue spaces are often called Nakano spaces. We refer
to Maligranda's paper \cite{M11} for the role of Hidegoro Nakano in the 
study of variable Lebesgue spaces. 

The  following property of the unit ball of variable Lebesgue spaces is well 
known (see, e.g., \cite[Corollary~2.22]{CF13}).
\begin{lemma}\label{le:VLS-unit-ball-property}
Let $\gamma$ be a measurable subset of a rectifiable Jordan curve $\Gamma$.
If $p\in\cP(\Gamma)$ and $f$ is a measurable function on $\gamma$, then
the inequalities $\varrho_{p(\cdot),\gamma}(f)\le 1$ and 
$\|f\|_{L^{p(\cdot)}(\gamma)}\le 1$ are equivalent.
\end{lemma}

For the brevity, we will simply 
write $L^{p(\cdot)}$ for $L^{p(\cdot)}(\Gamma)$.
For $p\in\cP(\Gamma)$, put
\[
p_-:=\operatornamewithlimits{ess\,inf}_{t\in\Gamma} p(t),
\quad
p_+:=\operatornamewithlimits{ess\,sup}_{t\in\Gamma} p(t).
\] 
\begin{lemma}[{\cite[Corollary~2.81]{CF13}}]
\label{le:reflexivity-VLS}
Let $\Gamma$ be a rectifiable Jordan curve and $p\in\cP(\Gamma)$. Then
$L^{p(\cdot)}$ is reflexive if and only if $1<p_-\le p_+<\infty$.
\end{lemma}
Embeddings of variable Lebesgue spaces are characterized as follows.
\begin{lemma}[{\cite[Corollary~2.48]{CF13}}]
\label{le:embeddings-VLS}
Let $\Gamma$ be a rectifiable Jordan curve. Suppose $p,q\in\cP(\Gamma)$. Then
$L^{p(\cdot)}\hookrightarrow L^{q(\cdot)}$ if and only if $q(t)\le p(t)$ for
almost all $t\in\Gamma$.
\end{lemma}
\subsection{Pointwise multipliers between variable Lebesgue spaces}
In this subsection we will describe the space of pointwise multipliers
between variable Lebesgue spaces. The next lemma follows
from \cite[Section~2, Property (f) and Theorem~1]{MP89} and the fact
that variable Lebesgue spaces are Banach function spaces 
\cite[Section~2.10.3]{CF13}.
\begin{lemma}\label{le:Maligranda-Persson}
Let $\gamma$ be a measurable subset of a rectifiable Jordan curve $\Gamma$
and $p\in\cP(\Gamma)$. Then 
\[
M(L^\infty(\gamma),L^{p(\cdot)}(\gamma))\equiv L^{p(\cdot)}(\gamma),
\quad
M(L^{p(\cdot)}(\gamma),L^{p(\cdot)}(\gamma))\equiv L^\infty(\gamma).
\]
\end{lemma}
Now we state the following two simple statements.
\begin{lemma}\label{le:VLS-direct-sum}
Let $\Gamma$ be a rectifiable Jordan curve and $\gamma_1,\dots,\gamma_k$
be measurable subsets of $\Gamma$ such that
\begin{equation}\label{eq:partition-Gamma}
\gamma_i\cap\gamma_j=\emptyset
\quad\mbox{for}\quad i,j\in\{1,\dots,k\},
\quad
\gamma_1\cup\dots\cup\gamma_k=\Gamma.
\end{equation}
If $p\in\cP(\Gamma)$, then
\[
L^{p(\cdot)}=L^{p(\cdot)}(\gamma_1)\oplus\dots\oplus L^{p(\cdot)}(\gamma_k),
\]
where the norm in the direct sum 
$L^{p(\cdot)}(\gamma_1)\oplus\dots\oplus L^{p(\cdot)}(\gamma_k)$
is defined by
\[
\|f\|_{L^{p(\cdot)}(\gamma_1)\oplus\dots\oplus L^{p(\cdot)}(\gamma_k)}
=
\|f\chi_{\gamma_1}\|_{L^{p(\cdot)}(\gamma_1)}
+
\dots
+
\|f\chi_{\gamma_k}\|_{L^{p(\cdot)}(\gamma_k)}.
\]
\end{lemma}
\begin{lemma}\label{le:VLS-multiplier-decomposition}
Let $\Gamma$ be a rectifiable Jordan curve and $\gamma_1,\dots,\gamma_k$
be measurable subsets of $\Gamma$ satisfying \eqref{eq:partition-Gamma}.
If $p,q\in\cP(\Gamma)$ and $q(t)\le p(t)$ for almost all $t\in\Gamma$, then
\begin{align*}
&
M\big(L^{p(\cdot)}(\gamma_1)\oplus \dots\oplus L^{p(\cdot)}(\gamma_k),
L^{q(\cdot)}(\gamma_1)\oplus\dots\oplus L^{q(\cdot)}(\gamma_k)\big)
\\
&=
M(L^{p(\cdot)}(\gamma_1),L^{q(\cdot)}(\gamma_1))
\oplus \dots \oplus
M(L^{p(\cdot)}(\gamma_k),L^{q(\cdot)}(\gamma_k)).
\end{align*}
\end{lemma}
The proofs of the above two lemmas are straightforward and they are omitted.

We will need the following generalized H\"older inequality.
\begin{lemma}[{\cite[Corollary~2.28]{CF13}}]
\label{le:Hoelder-VLS}
Let $\Gamma$ be a rectifiable Jordan curve. Suppose $p,q,r\in\cP(\Gamma)$
are related by
\begin{equation}\label{eq:Hoelder}
\frac{1}{q(t)}=\frac{1}{p(t)}+\frac{1}{r(t)},
\quad t\in\Gamma.
\end{equation}
Then there exists a constant $C>0$ such that for all 
$f\in L^{p(\cdot)}$ and $g\in L^{r(\cdot)}$, one has
$fg\in L^{q(\cdot)}$ and
\[
\|fg\|_{L^{q(\cdot)}}\le C\|f\|_{L^{p(\cdot)}}\|g\|_{L^{r(\cdot)}}.
\]
\end{lemma}
The following result was obtained by Nakai \cite[Example~4.1]{N16}
under the additional hypothesis 
\[
\sup_{t\in\Gamma\setminus \Gamma_\infty^{r(\cdot)}}r(t)<\infty
\]
(and in the more general setting of quasi-Banach variable 
Lebesgue spaces spaces over arbitrary measure spaces).
Nakai also mentioned in \cite[Remark~4.2]{N16} (without proof) that this 
hypothesis is superfluous. For the convenience of the reader, we provide a 
proof here.
\begin{theorem}\label{th:multiplier-space-VLS}
Let $\Gamma$ be a rectifiable Jordan curve. Suppose $p,q,r\in\cP(\Gamma)$ are 
related by \eqref{eq:Hoelder}. Then 
$M(L^{p(\cdot)},L^{q(\cdot)})=L^{r(\cdot)}$.
\end{theorem}
\begin{proof}
Let $\gamma_1:=\Gamma_\infty^{p(\cdot)}$, 
$\gamma_2:=(\Gamma_\infty^{q(\cdot)}
\cup\Gamma_\infty^{r(\cdot)})\setminus\Gamma_\infty^{p(\cdot)}$, and
\[
\gamma_3:=\Gamma\setminus(\gamma_1\cup\gamma_2)
=
\Gamma\setminus
(\Gamma_\infty^{p(\cdot)}
\cup
\Gamma_\infty^{q(\cdot)}
\cup
\Gamma_\infty^{r(\cdot)}).
\]

From \eqref{eq:Hoelder} it follows that $p(t)=\infty$ and $q(t)=r(t)$
for $t\in\gamma_1$. Then by Lemma~\ref{le:Maligranda-Persson},
\begin{equation}\label{eq:multiplier-space-VLS-1}
M(L^{p(\cdot)}(\gamma_1),L^{q(\cdot)}(\gamma_1))
\equiv
M(L^\infty(\gamma_1),L^{r(\cdot)}(\gamma_1))
\equiv
L^{r(\cdot)}(\gamma_1).
\end{equation}

Similarly, from \eqref{eq:Hoelder} we also obtain 
$\Gamma_\infty^{q(\cdot)}\subset
\Gamma_\infty^{p(\cdot)}\cap\Gamma_\infty^{r(\cdot)}$, whence
$\gamma_2=\Gamma_\infty^{r(\cdot)}\setminus\Gamma_\infty^{p(\cdot)}$.
Therefore, $p(t)=q(t)<\infty$ and $r(t)=\infty$ for $t\in\gamma_2$.
Then, from Lemma~\ref{le:Maligranda-Persson} we get
\begin{equation}\label{eq:multiplier-space-VLS-2}
M(L^{p(\cdot)}(\gamma_2),L^{q(\cdot)}(\gamma_2))
\equiv
M(L^{p(\cdot)}(\gamma_2),L^{p(\cdot)}(\gamma_2))
\equiv
L^\infty(\gamma_2)
\equiv
L^{r(\cdot)}(\gamma_2).
\end{equation}

The rest of the proof is developed by analogy with the proof
of \cite[Theorem~4]{MP89}. Let 
$f\in M(L^{p(\cdot)}(\gamma_3),L^{q(\cdot)}(\gamma_3))$.
The multiplication operator $Tg=fg$ maps $L^{p(\cdot)}(\gamma_3)$
into $L^{q(\cdot)}(\gamma_3)$ and has a closed graph. Hence there
exists a constant $c\in(0,\infty)$ such that
\begin{equation}\label{eq:multiplier-space-VLS-3}
\|fg\|_{L^{q(\cdot)}(\gamma_3)}\le c\|g\|_{L^{p(\cdot)}(\gamma_3)}
\quad\mbox{for all}\quad
g\in L^{p(\cdot)}(\gamma_3).
\end{equation}
For $\eps>0$, put
\begin{equation}\label{eq:multiplier-space-VLS-4}
f_\eps(t)=
\left\{\begin{array}{lll}
\displaystyle
\frac{c+\eps}{f(t)}\left(\frac{|f(t)|}{c+\eps}\right)^{r(t)/q(t)}
&\mbox{if}& f(t)\ne 0,
\\[3mm]
0, &\mbox{if}& f(t)=0.
\end{array}\right.
\end{equation}
Let us show that
\begin{equation}\label{eq:multiplier-space-VLS-5}
\varrho_{p(\cdot),\gamma_3}(f_\eps)\le 1.
\end{equation}
Assume the contrary, that is, $\varrho_{p(\cdot),\gamma_3}(f_\eps)> 1$.
Then from \cite[Propositions A.1 and A.8]{DN11} it follows that
there exists a measurable set $\gamma\subset\gamma_3$ such that
\begin{equation}\label{eq:multiplier-space-VLS-6}
\varrho_{p(\cdot),\gamma_3}(\chi_\gamma f_\eps)= 1.
\end{equation}
From \eqref{eq:Hoelder} and \eqref{eq:multiplier-space-VLS-4} we get
\begin{equation}\label{eq:multiplier-space-VLS-7}
|f_\eps(t)|=\left(\frac{|f(t)|}{c+\eps}\right)^{r(t)/q(t)-1}=
\left(\frac{|f(t)|}{c+\eps}\right)^{r(t)/p(t)},
\quad
t\in\gamma.
\end{equation}
Equality \eqref{eq:multiplier-space-VLS-6} and 
Lemma~\ref{le:VLS-unit-ball-property} imply that 
$\|\chi_\gamma f_\eps\|_{L^{p(\cdot)}(\gamma_3)}\le 1$. Applying
\eqref{eq:multiplier-space-VLS-3} with $g=\chi_\gamma f_\eps$, we obtain
\[
\left\|\frac{\chi_\gamma f_\eps f}{c}\right\|_{L^{q(\cdot)}(\gamma_3)}
\le
\|\chi_\gamma f\|_{L^{p(\cdot)}(\gamma_3)}\le 1.
\]
Then, in view of Lemma~\ref{le:VLS-unit-ball-property}, we get
\begin{equation}\label{eq:multiplier-space-VLS-8}
\varrho_{q(\cdot),\gamma_3}\left(\frac{\chi_\gamma f_\eps f}{c}\right)\le 1.
\end{equation}
Combining \eqref{eq:multiplier-space-VLS-6}, \eqref{eq:multiplier-space-VLS-4},
\eqref{eq:Hoelder}, and \eqref{eq:multiplier-space-VLS-8}, we arrive at
\begin{align*}
1&=
\varrho_{p(\cdot),\gamma_3}(\chi_\gamma f_\eps)
=
\varrho_{r(\cdot),\gamma_3}\left(\frac{\chi_\gamma f}{c+\eps}\right)
=
\varrho_{q(\cdot),\gamma_3}\left(\frac{\chi_\gamma f_\eps f}{c+\eps}\right)
\\
&\le 
\frac{c}{c+\eps}
\varrho_{q(\cdot),\gamma_3}\left(\frac{\chi_\gamma f_\eps f}{c}\right)
\le 
\frac{c}{c+\eps}
<1,
\end{align*}
and we get a contradiction. Hence \eqref{eq:multiplier-space-VLS-5} is
fulfilled. Applying Lemma~\ref{le:VLS-unit-ball-property} to
\eqref{eq:multiplier-space-VLS-5}, 
we deduce that $\|f_\eps\|_{L^{p(\cdot)}(\gamma_3)}\le 1$. Then,
in view of \eqref{eq:multiplier-space-VLS-3}, we obtain
\[
\|f_\eps f\|_{L^{q(\cdot)}(\gamma_3)}
\le 
c\|f_\eps\|_{L^{p(\cdot)}(\gamma_3)}\le c.
\]
Taking into account the above inequality, equality
\eqref{eq:multiplier-space-VLS-4} and Lemma~\ref{le:VLS-unit-ball-property},
we see that
\[
\varrho_{r(\cdot),\gamma_3}\left(\frac{f}{c+\eps}\right)
=
\varrho_{q(\cdot),\gamma_3}\left(\frac{f_\eps f}{c+\eps}\right)
\le
\varrho_{q(\cdot),\gamma_3}\left(\frac{f_\eps f}{c}\right)
\le 1,
\]
whence $\|f\|_{L^{r(\cdot)}(\gamma_3)}\le c+\eps$. Letting $\eps\to 0$,
we obtain $\|f\|_{L^{r(\cdot)}(\gamma_3)}\le c$. It remains to observe
that the smallest constant in inequality \eqref{eq:multiplier-space-VLS-3}
coincides with $\|f\|_{M(L^{p(\cdot)}(\gamma_3),L^{q(\cdot)}(\gamma_3))}$.
Hence 
\[
M(L^{p(\cdot)}(\gamma_3),L^{q(\cdot)}(\gamma_3))\hookrightarrow 
L^{r(\cdot)}(\gamma_3). 
\]
The embedding
\[
L^{r(\cdot)}(\gamma_3)\hookrightarrow 
M(L^{p(\cdot)}(\gamma_3),L^{q(\cdot)}(\gamma_3))
\]
follows from the generalized H\"older inequality (Lemma~\ref{le:Hoelder-VLS}).
Thus,
\begin{equation}\label{eq:multiplier-space-VLS-9}
M(L^{p(\cdot)}(\gamma_3),L^{q(\cdot)}(\gamma_3))
=
L^{r(\cdot)}(\gamma_3).
\end{equation}

Finally, from \eqref{eq:multiplier-space-VLS-1}, 
\eqref{eq:multiplier-space-VLS-2}, \eqref{eq:multiplier-space-VLS-9}
and Lemmas~\ref{le:VLS-direct-sum}--\ref{le:VLS-multiplier-decomposition}
we obtain
\begin{align*}
&
M(L^{p(\cdot)},L^{q(\cdot)})
=
\\
&=
M\big(
L^{p(\cdot)}(\gamma_1)
\oplus 
L^{p(\cdot)}(\gamma_2)
\oplus 
L^{p(\cdot)}(\gamma_3),
L^{q(\cdot)}(\gamma_1)
\oplus 
L^{q(\cdot)}(\gamma_2)
\oplus 
L^{q(\cdot)}(\gamma_3)
\big)
\\
&=
M(L^{p(\cdot)}(\gamma_1),L^{q(\cdot)}(\gamma_1))
\oplus
M(L^{p(\cdot)}(\gamma_2),L^{q(\cdot)}(\gamma_2))
\oplus
M(L^{p(\cdot)}(\gamma_3),L^{q(\cdot)}(\gamma_3))
\\
&=
L^{r(\cdot)}(\gamma_1)
\oplus
L^{r(\cdot)}(\gamma_2)
\oplus
L^{r(\cdot)}(\gamma_3)=L^{r(\cdot)},
\end{align*}
which completes the proof.
\end{proof}
The above proof can be extended without any change to the case of variable 
Lebesgue spaces over arbitrary nonatomic measure spaces. The theorem itself 
is also true for arbitrary measure spaces. However the proof for not 
necessarily nonatomic measure spaces is more complicated. It can be developed 
by analogy with \cite{MN10}.
\subsection{The Cauchy singular integral operator $S$ on variable Lebesgue spaces}
David's theorem \cite{D84} (see also \cite[Theorem~4.17]{BK97}), says that the 
Cauchy singular integral operator $S$ is bounded on the standard Lebesgue space 
$L^p$, $1<p<\infty$, over a rectifiable Jordan curve $\Gamma$ if and only if 
$\Gamma$ is a Carleson curve. To formulate the generalization of this result
to the setting of variable Lebesgue spaces, we will need the following class
of nice variable exponents. 

Let $\Gamma$ be a rectifiable Jordan curve. We say that an exponent 
$p\in\cP(\Gamma)$ is locally log-H\"older continuous (cf. 
\cite[Definition~2.2]{CF13}) if $1<p_-\le p_+<\infty$
and there exists a constant $C_{p(\cdot),\Gamma}\in(0,\infty)$ such that
\[
|p(t)-p(\tau)|\le\frac{C_{p(\cdot),\Gamma}}{-\log|t-\tau|}
\quad\mbox{for all}\quad t,\tau\in\Gamma\quad\mbox{satisfying}\quad |t-\tau|<1/2.
\]
The class of all locally log-H\"older continuous exponent will be denoted by
$LH(\Gamma)$. Notice that some authors also denote this class by 
$\mathbb{P}^{\log}(\Gamma)$, see, e.g., \cite[Section~1.1.4]{KMRS16}.
\begin{theorem}[{\cite[Theorems~2.45 and 2.49]{KMRS16}}]
\label{th:S-boundedness-VLS}
Let $\Gamma$ be a rectifiable Jordan curve and $p\in LH(\Gamma)$.
Then the Cauchy singular integral operator $S$ is bounded on $L^{p(\cdot)}$
if and only if $\Gamma$ is a Carleson curve. 
\end{theorem}
\subsection{The Coburn-Simonenko theorem for Toeplitz operators acting
between Hardy type subspaces of variable Lebesgue spaces}
Now we are in a position to give a more precise formulation of 
Theorem~\ref{th:main} in the case of Toeplitz operators acting between
Hardy type subspaces $PL^{p(\cdot)}$ and $PL^{q(\cdot)}$ of variable
Lebesgue spaces $L^{p(\cdot)}$ and $L^{q(\cdot)}$, respectively.
\begin{theorem}\label{th:main-to-VLS}
Let $\Gamma$ be a Carleson Jordan curve. Suppose variable
exponents $p,q\in LH(\Gamma)$ and $r\in\cP(\Gamma)$ are related by
\eqref{eq:Hoelder}. If $a\in L^{r(\cdot)}\setminus\{0\}$, then the Toeplitz 
operator $T(a)\in\cL(PL^{p(\cdot)},PL^{q(\cdot)})$ has a trivial kernel 
in $PL^{p(\cdot)}$ or a dense image in $PL^{q(\cdot)}$.
\end{theorem} 
\begin{proof}
We know from Lemma~\ref{le:reflexivity-VLS} that the spaces $L^{p(\cdot)}$
and $L^{q(\cdot)}$ are reflexive because $1<p_-,q_-$ and $p_+,q_+<\infty$
(in view of $p,q\in LH(\Gamma)$). Since $r\in\cP(\Gamma)$, we have 
$1\le r(t)\le\infty$ for almost all $t\in\Gamma$. Then we deduce from 
\eqref{eq:Hoelder} that $q(t)\le p(t)$ for almost all $t\in\Gamma$. Therefore, 
by Lemma~\ref{le:embeddings-VLS}, $L^{p(\cdot)}\hookrightarrow L^{q(\cdot)}$.
It follows from Theorem~\ref{th:S-boundedness-VLS} that the Cauchy singular 
integral operator $S$ is bounded on $L^{p(\cdot)}$ and $L^{q(\cdot)}$.
Now we observe that $L^{r(\cdot)}=M(L^{p(\cdot)},L^{q(\cdot)})$
in view of Theorem~\ref{th:multiplier-space-VLS}. It remains to apply 
Theorem~\ref{th:main}.
\end{proof}
Corollary~\ref{co:Hardy} follows immediately from Theorem~\ref{th:main-to-VLS}
if we take all exponents $p,q$, and $r$ constant.
\subsection*{Acknowledgement}
I would like to thank Karol Le\'snik for stimulating discussions and for 
sharing with me a preliminary version of \cite{L17}.


\begin{thebibliography}{00}
\bibitem{BS88}
C. Bennett and R. Sharpley,
\textit{Interpolation of Operators},
Pure and Applied Mathematics, \textbf{129}.
Academic Press, Boston, 1988.

\bibitem{BK97}
A. B\"ottcher and Yu. I. Karlovich,
\textit{Carleson Curves, Muckenhoupt Weights, and Toeplitz Operators},
Birkh\"auser, Basel, 1997.

\bibitem{BS06}
A. B\"ottcher and B. Silbermann,
\textit{Analysis of Toeplitz Operators},
2nd edition. Springer, Berlin, 2006.

\bibitem{C66}
L. A. Coburn,
\textit{Weyl's theorem for non-normal operators},
Michigan Math. J., \textbf{13} (1966), 285--286.

\bibitem{CF13}
D. Cruz-Uribe and A. Fiorenza,
\textit{Variable Lebesgue Spaces}, Birkh\"auser, Basel, 2013.

\bibitem{D84}
G. David, 
\textit{Op{\'e}rateurs int{\'e}graux singuliers sur certaines courbes du plan complexe},
Ann. Sci. Ec. Norm. Super. (4) \textbf{17} (1984), 157--189.

\bibitem{D98}
R. G. Douglas,
\textit{Banach Algebra Techniques in Operator Theory},
2nd edition. Springer, New York, 1998.

\bibitem{DN11}
R. M. Dudley and R. Norvai\v{s}a, 
\textit{Concrete Functional Calculus},
Springer, New York, 2011.

\bibitem{D70}
P. L. Duren, 
\textit{Theory of $H^p$ spaces},
Academic Press, New York, 1970.

\bibitem{Dynkin87}
E. M. Dynkin,
\textit{Methods of the theory of singular integrals (Hilbert transform and
Calder\'on-Zygmund theory)},
Itogi nauki i tehniki VINITI, Ser. Sovrem. probl. mat.,
\textbf{15} (1987), 197--292 (in Russian).
English translation:
Commutative harmonic analysis I. General survey. Classical aspects, 
Encycl. Math. Sci., \textbf{15} (1991), 167--259.

\bibitem{GGK90}
I.~Gohberg, S.~Goldberg, and M.~Kaashoek,
\textit{Classes of Linear Operators}. Vol. II,
Birkh\"auser Verlag, Basel, 1990.

\bibitem{KA84}
L. V. Kantorovich and G. P. Akilov,
\textit{Functional Analysis},
Nauka, Moscow, 3rd ed., 1984 (in Russian).
English translation: Pergamon Press, Oxford, 2nd ed., 1982.

\bibitem{K03}
A. Yu. Karlovich,
\textit{Fredholmness of singular integral operators with piecewise continuous 
coefficients on weighted Banach function spaces},
J. Integral Equations Appl. \textbf{15} (2003), 263--320.

\bibitem{K04}
A. Yu. Karlovich,
\textit{Norms of Toeplitz and Hankel operators on Hardy type subspaces 
of rearrangement-invariant spaces},
Integr. Equ. Oper. Theory \textbf{49} (2004), 43--64.

\bibitem{KMRS16}
V. Kokilashvili, A. Meskhi, H. Rafeiro, and S. Samko, 
\textit{Integral Operators in Non-Standard Function Spaces. 
Volume 1: Variable Exponent Lebesgue and Amalgam Spaces}.
Birkh\"auser Verlag, Basel, 2016.

\bibitem{KLM13}
P. Kolwicz, K. Le{\'s}nik, and L. Maligranda,
\textit{Pointwise multipliers of Calder{\'o}n-Lozanovskii spaces},
Math. Nachr. \textbf{286} (2013), 876--907.

\bibitem{L17}
K. Le\'snik,
\textit{Toeplitz and Hankel operators between distinct Hardy spaces},
arXiv:1708.00910 [math.FA].

\bibitem{LT16}
K. Le\'snik and  J. Tomaszewski,
\textit{Pointwise multipliers of Orlicz function spaces and factorization},
Positivity, to appear.
DOI: 10.1007/s11117-017-0485-x

\bibitem{M11}
L. Maligranda,
\textit{Hidegoro Nakano (1909-1974) -- on the centenary of his birth},
M. Kato, (ed.) et al., 
Proceedings of the 3rd international symposium on Banach and function spaces (ISBFS 2009), 
Kitakyushu, Japan, September 14--17, 2009. Yokohama, Yokohama Publishers, pp.~99--171, 2011.

\bibitem{MN10}
L. Maligranda and E. Nakai, 
\textit{Pointwise multipliers of Orlicz spaces}, 
Arch. Math. \textbf{95} (2010), 251-–256.

\bibitem{MP89}
L. Maligranda and L. E. Persson, 
\textit{Generalized duality of some Banach function spaces},
Indag. Math. \textbf{51} (1989), 323--338.

\bibitem{N16}
E. Nakai, 
\textit{Pointwise multipliers on Musielak-Orlicz spaces}, 
Nihonkai Math. J. \textbf{27} (2016), 135--146.

\bibitem{N02}
N. K. Nikolski,
\textit{Operators, Functions, and Systems: an Easy Reading. 
Volume I: Hardy, Hankel, and Toeplitz},
American Mathematical Society, Providence, RI, 2002.

\bibitem{P50}
I. I. Privalov,
\textit{Boundary Properties of Analytic Functions},
Gosudarstv. Izdat. Tehn.-Teor. Lit., Moscow-Leningrad, 1950 (in Russian).

\bibitem{R91}
W. Rudin, 
\textit{Functional Analysis}, 2nd edition.
McGraw-Hill, New York, NY, 1991.

\bibitem{S68}
I. B. Simonenko,
\textit{Some general questions in the theory of the Riemann boundary problem},
Math. USSR Izvestiya \textbf{2} (1968) 1091--1099.

\bibitem{T87} 
V. A. Tolokonnikov, 
\textit{Hankel and Toeplitz operators in Hardy spaces}, 
J. Soviet Math. \textbf{37} (1987), 1359--1364.
\end{thebibliography}
\end{document}